\documentclass{amsart}

\usepackage{enumitem}
\usepackage{mathrsfs}
\usepackage{amsmath,hyperref}

 \newcommand{\GI}{\operatorname{G}}
\newcommand{\II}{\operatorname{I}}
\newcommand{\OI}{\operatorname{O}}

\DeclareMathOperator{\Ker}{Ker}
\DeclareMathOperator{\img}{Im}

\DeclareMathOperator{\codim}{codim}
\DeclareMathOperator{\Ll}{span}
\newcommand{\agi}{algebraic generalized inverse}

\newcommand{\oi}{outer inverse}
\newcommand{\ii}{inner inverse}
\newcommand{\id}{1}

\newcommand{\C}{\mathscr{C}}
\newcommand{\B}{\mathscr{B}}
\newcommand{\R}{\mathbb{R}}
\newcommand{\Q}{\mathbb{Q}}

\newcommand{\es}{E}
\newcommand{\Bo}{\B^{\perp}}

\newcommand{\Boi}[1]{\B_{#1}^{\perp}}

\newcommand{\PG}{\mathbb{P}}

\newcommand{\E}{\operatorname{E}}
\newcommand{\D}{\operatorname{D}}
\newcommand{\pack}{\texttt{IntDiffOp}}
\newcommand{\maple}{\textsc{Maple}}
\newcommand{\dx}{\,\mathrm{d}}
\newcommand{\bvp}[2]{\boxed{\begin{array}{l}#1\\#2\end{array}}}
\newcommand{\scal}[2]{\langle #1 , #2 \rangle}

\newtheorem{theorem}{Theorem}
\newtheorem{proposition}[theorem]{Proposition}
\newtheorem{lemma}[theorem]{Lemma}
\newtheorem{corollary}[theorem]{Corollary}
\newtheorem{definition}[theorem]{Definition}

\newtheorem{remark}[theorem]{Remark}

\begin{document}

\title[On the product of projectors and generalized inverses]{On the product of \\ projectors and generalized inverses}

\author{Anja Korporal}
\address{Johann Radon Institute for Computational and Applied Mathematics (RICAM) \\ Austrian Academy of Sciences \\ A-4040 Linz, Austria}
\curraddr{}
\email{anja.korporal@ricam.oeaw.ac.at}
\thanks{}

\author{Georg Regensburger}
\address{Johann Radon Institute for Computational and Applied Mathematics (RICAM) \\ Austrian Academy of Sciences \\ A-4040 Linz, Austria}
\curraddr{}
\email{georg.regensburger@ricam.oeaw.ac.at}
\thanks{}

\subjclass[2000]{Primary 15A09; Secondary 47A05}

\keywords{Generalized inverse, projector, reverse order law, Fredholm operator, linear boundary problem, duality}

\date{\today}

\maketitle

\begin{abstract}

We consider generalized inverses 
of linear operators 
on arbitrary vector spaces and study the question when 
their product in reverse order is again a generalized inverse. 
This problem is equivalent to the question when the product of two projectors is again a
projector, and we discuss necessary and sufficient conditions in terms of their kernels and images alone. 
We give a new representation of the product of generalized inverses
that does not require explicit knowledge of the factors. 
Our approach is based on implicit representations of subspaces via their
orthogonals in the dual space.
For Fredholm operators, the corresponding computations
reduce to finite-dimensional problems.
 We illustrate our results with examples for matrices and linear ordinary boundary problems.
\end{abstract}

\section{Introduction}
Analogues of the reverse order law  $(AB)^{-1} = B^{-1}A^{-1}$ for bijective operators 
have been studied intensively for various kinds of generalized inverses. 
Most articles and books are concerned with the matrix case;
see for example 
\cite{Greville1966, Erdelyi1966, RaoMitra1971, ShinozakiSibuya1974, Werner1994, DePierroWei1998, Ben-IsraelGreville2003, Mitra2010, TakaneTianYanai2007, TianCheng2004, LiuWei2008}.
For infinite-dimensional vector spaces, usually additional topological structures like Banach or Hilbert spaces 
are assumed; see for example \cite{Nashed1987, DjordjevicRakovevic2008, DjordjevicDincic2010, DincicDjordjevic2013}. In our approach, we systematically exploit duality results that hold in arbitrary vector spaces
and a corresponding duality principle for statements about generalized inverses and projectors; see
Appendix~(\ref{sec:Duality}).

The validity of the reverse order law can be reduced to the question whether the product of
two projectors is a projector (Section \ref{sec:GI}).
This problem is studied  in \cite{GrossTrenkler1998, TakaneYanai1999, Werner1992} for finite-dimensional vector spaces.
We discuss necessary and sufficient conditions that carry over to arbitrary vector spaces and 
can be expressed in terms of the kernels and images of the respective operators alone (Section \ref{sec:Projectors}).
Applying the duality principle leads to new conditions and a characterization of the 
commutativity of two projectors that generalizes a result from \cite{BaksalaryBaksalary2002}.

In Section~\ref{sec:ROL}, we translate the results for projectors to generalized inverses and obtain necessary and sufficient conditions for the reverse order law in arbitrary vector spaces. Based on these conditions, we give a short proof for the characterization in Theorem~\ref{thm:allinner} of two operators such that the reverse holds for all inner inverses (also called g-inverses or $\{1\}$-inverses). Moreover, we show
that there always exist \agi s (also called $\{1,2\}$-inverses) of two operators $A$ and $B$ such that 
their product in reverse order is an \agi\ of $AB$.

Assuming the reverse order law to hold, Theorem \ref{thm:RevOrderLaw} gives a representation of 
the product of two outer inverses ($\{2\}$-inverses) that can be computed 
using only kernel and image of the outer inverses of the factors. In this representation, we rely on a description of the kernel of a composition using inner inverses (Section \ref{sec:Compositions}) and
implicit representations of subspaces via their orthogonals in the dual space.
Moreover, we avoid 
 the computation of generalized inverses by using the associated transpose map. Examples for matrices illustrating the results are given in Section~\ref{sec:Example}.

An important application for our results is given by linear boundary problems (Section~\ref{sec:BP}).
Their solution operators (Green's operators) are generalized inverses, and it is natural to express infinite
dimensional solution spaces implicitly via the (homogeneous) boundary conditions they satisfy.
 Green's operators for ordinary boundary problems are Fredholm operators, for which we can check the conditions for 
the reverse order law algorithmically and compute the implicit representation of the product (Section \ref{sec:Fredholm}). 
Hence we can test if the product of two (generalized) Green's operators is again a Green's operator, and we can determine 
which boundary problem it solves.

\section{Generalized inverses} \label{sec:GI}

 In this section, we first recall basic properties of generalized inverses. 
For further details and proofs, we refer to \cite{NashedVotruba1976, Nashed1987}
and the references therein.
Throughout this article, $U$, $V$, and $W$ always denote vector spaces over the same field $F$, and we use the notation $V_1 \leq V$ for a subspace $V_1$ of $V$.
\begin{definition}\label{dfn:GI}
   Let $T\colon V \to W$ be linear. We call a linear map $G \colon W \to V$ an \emph{inner inverse} of $T$ if $TGT = T$
   and an \emph{outer inverse} of $T$ if $GTG = G$. If $G$ is an inner and an outer inverse of $T$, we call $G$
   an \emph{\agi}\ of $T$.
\end{definition}
This terminology of generalized inverses is adopted from \cite{NashedVotruba1976}; 
other sources refer to inner inverses as generalized inverses or 
g-inverses, whereas \agi s are also called reflexive generalized inverses. Also the notations $\{1\}$-inverse
(resp. $\{2\}$- and $\{1,2\}$-inverse) are used, which refer to the corresponding Moore-Penrose equations the generalized inverse satisfies.

\begin{proposition} 
\label{prop:OIChar}
 Let $T \colon V \to W$ and $G \colon W \to V$ be linear. The following statements are equivalent:
  \begin{enumerate}[label=(\roman*)]
  \item $G$ is an outer inverse of $T$. 
  \item $GT$ is a projector and $\img GT = \img G$. \label{oi:Nashed2}
  \item $GT$ is a projector and $V = \img G \oplus \Ker GT$. \label{oi:Nashed3}
  \item $GT$ is a projector and $W = \img T + \Ker G$. \label{oi:Nashed4}
  \item $TG$ is a projector and $\Ker TG = \Ker G$. \label{oi:Nashed5} 
  \item $TG$ is a projector and $W = \Ker G \oplus \img TG$. \label{oi:Nashed6} 
  \item $TG$ is a projector and $\img G \cap \Ker T = \{0\}$. \label{oi:Nashed7}
 \end{enumerate}
\end{proposition}

Corresponding to \ref{oi:Nashed7} and \ref{oi:Nashed6}, for subspaces $B \leq V$ and $E \leq W$ with
 \begin{equation*} 
 B \cap \Ker T = \{0\} \quad \text{and} \quad 
 W = E \oplus T(B),
\end{equation*}
we can construct an \oi\ $ G$ of $T$ with 
 $\img G = B$ and $\Ker G = E$
 as follows; cf.\ \cite[Cor.\ 8.2]{Nashed1987}.
We consider the projector $Q$ with
\begin{equation} \label{eq:DefQOI}
 \img Q  = T(B), \;\; \Ker Q = E.
\end{equation}
The restriction $T|_B\colon B \to  T(B)$ is bijective since $B \cap \Ker T = \{0\}$, and we can define
 $ G = (T|_B)^{-1}Q$. One easily verifies that $G$ is an \oi\ 
of $T$ with $\img G = B$ and $\Ker G = E$. 
Since by \ref{oi:Nashed3} 
we have $V = B \oplus T^{-1}(E)$,  
we define the projector $P$ in analogy to $Q$ 
by
\begin{equation} \label{eq:DefPOI}
  \img P = T^{-1}(E), \;\; \Ker P = B. 
\end{equation}  
Then, by definition and by
Proposition \ref{prop:OIChar}, we have
\begin{equation*}\label{eq:Identities}
 GTG = G, \quad
 TG = Q, \quad \text{and} \quad
 GT = \id - P, 
\end{equation*}
and $G$ is determined uniquely by these equations. Hence an \oi\ depends only on the choice of the defining spaces 
 $B$ and $E$. We use the notations
 $G = \OI(T, B, E)$ and 
 $G = \OI(T, P, Q)$ 
for $P$ and $Q$ as in \eqref{eq:DefPOI} and \eqref{eq:DefQOI}.

Obviously, $G$ is an outer inverse of $T$ if and only if $T$ is an inner inverse of $G$. Therefore, we get a
result analogous to Proposition \ref{prop:OIChar} for inner inverses by interchanging the role of $T$ and $G$. The construction of \ii s is not completely analogous to outer inverses, see \cite[Prop.\ 1.3]{NashedVotruba1976}. 
For subspaces $B \leq V$ and $E \leq W$ 
such that  
\begin{equation} \label{eq:DirSumDec}
 V= \Ker T \oplus B\quad \text{and} \quad W = \img T  \oplus E,
\end{equation}
an inner inverse $G$ of $T$ is given on  $\img T$ by $(T|_B)^{-1}$ and can be chosen arbitrarily on $E$.
For such an inner inverse with  $B=\img GT$ and $E= \Ker TG$, we write $G \in \II(T, B, E)$.

For constructing \agi s,
we start with direct sums as in \eqref{eq:DirSumDec},
but require $\Ker G = E$ and $\img G = B$. We use the notation $G = \GI(T, B, E)$.

The following result for inner inverses 
is well-known  in the matrix case \cite{Searle1971,ShinozakiSibuya1974,Werner1992} and its elementary proof remains valid  for arbitrary vector spaces. 

\begin{proposition} \label{prop:CompIINS}
 Let $T_1 \colon V \to W$ and $T_2 \colon U \to V$ be linear with outer (resp. inner) inverses $G_1$ and $G_2$. 
 Let $P=G_1T_1$ and $Q=T_2G_2$. Then
 $G_2G_1$ is an outer (resp. inner) inverse of $T_1T_2$ if and only if $QP$ (resp. $PQ$) is a projector.
\end{proposition}

\begin{proof}
Let $G_2G_1$ be an outer inverse of $T_1T_2$, that is, $G_2G_1 = G_2G_1T_1T_2G_2G_1$. 
Multiplying with $T_2$ from the left and with $T_1$ from the right yields
\begin{equation*} 
 T_2G_2G_1T_1 = T_2G_2G_1T_1T_2G_2G_1T_1,  
\end{equation*}
thus $QP = T_2G_2G_1T_1$ is a projector. For the other direction, we multiply the
previous equation with $G_2$ from the left and $G_1$ from the right and 
use that $G_1T_1G_1 =G_1$ and $G_2T_2G_2=G_2$.
The proof for inner inverses follows  by interchanging the roles of $T_i$ and $G_i$.
\end{proof}

\section{Kernel of compositions} \label{sec:Compositions}
 
We now describe the inverse image of a subspace under the composition of two linear maps
using inner inverses. For projectors, kernel and image of the composition can be expressed in
terms of kernel and image of the corresponding factors alone. Note that a projector is an inner inverse of itself.

\begin{proposition} \label{prop:KerComp}
 Let $T_1 \colon V \to W$ and $T_2 \colon U \to V$ be linear and $G_2$ an inner inverse of $T_2$. For a subspace $W_1 \leq W$,
 we have	
\begin{equation*}
 (T_1T_2)^{-1}(W_1) = G_2(T_1^{-1}(W_1) \cap\img T_2) \oplus \Ker T_2
\end{equation*}
for the inverse image of the composition. In particular,
\begin{equation*} 
 \Ker T_1T_2 = G_2(\Ker  T_1 \cap \img T_2) \oplus \Ker T_2.
\end{equation*}
\end{proposition}

\begin{proof}
 Since $T_2G_2$ is a projector onto $\img T_2$ by Proposition \ref{prop:OIChar} \ref{oi:Nashed2} (interchanging the role of $T$ and $G$), we have
\begin{multline*}
 T_1T_2(G_2(T_1^{-1}(W_1)\cap \img T_2) + \Ker T_2) = T_1 Q_2 (T_1^{-1}(W_1) \cap \img T_2) + 0\\
= T_1 (T_1^{-1}(W_1) \cap \img T_2) \leq W_1 \cap \img T_1T_2 \leq W_1.
\end{multline*}
Conversely, let $ u\in (T_1T_2)^{-1}(W_1)$. Then $T_2u =v$ with $v \in T_1^{-1}(W_1)$. Since also $v \in \img T_2$, we have
\begin{equation*}
 T_2(u -G_2v) = T_2u - Q_2v = T_2 u -v = v-v= 0,
\end{equation*}
that is, $u - G_2v \in \Ker T_2$. Writing $u=G_2v+ u-G_2v$ yields $u \in G_2(T_1^{-1}(W_1) \cap\img T_2) + \Ker T_2$. 
The sum is direct since by Proposition \ref{prop:OIChar} \ref{oi:Nashed6} (interchanging the role of $T$ and $G$), we have $U = \Ker T_2 \oplus \img G_2T_2$.
\end{proof}

\begin{corollary} \label{cor:KerImComp}
 Let $T \colon V \to W$ be linear and let $P \colon V \to V$ and $Q \colon W \to W$ be projectors.
 Then
\begin{equation*}
 \Ker TQ  = (\Ker T \cap \img Q) \oplus \Ker Q \quad \text{and} \quad  \img PT  = (\img T + \Ker P) \cap \img P.
\end{equation*}
\end{corollary}

\begin{proof}
 Applying Proposition \ref{prop:KerComp} yields
\begin{equation*}
 \Ker TQ = Q(\Ker T \cap \img Q) \oplus \Ker Q = (\Ker T \cap \img Q) \oplus \Ker Q.
\end{equation*}
 The statement for the image follows from the duality principle \ref{DualPrin}.
\end{proof}

This result generalizes \cite[Lemma~2.2]{Werner1992}, where the kernel and image of
a product $PQ$ of two projectors are computed as above, when  $PQ$ is again a projector.

\section{Products of projectors} \label{sec:Projectors}
In view of Proposition \ref{prop:CompIINS}, we study  necessary and sufficient conditions
for the product of two projectors to be a projector.
Throughout this section let
 $P, Q \colon V \to V$
denote projectors.

The first of the following necessary and sufficient conditions for the product of $P$ and $Q$ to be a projector is mentioned
as an exercise without proof in 
\cite[p.\ 339]{BrownPage1970}.  
In \cite[Lemma~3]{GrossTrenkler1998} the same result is formulated for matrices but the
proof is valid for arbitrary vector spaces. The second necessary and sufficient condition for the matrix case is given in \cite[Lemma~2.2]{Werner1992}. 
The simpler proof from \cite{TakaneYanai1999} carries over to arbitrary vector spaces.

\begin{lemma} 
The composition $PQ$ is a projector if and only if
\[\img PQ \leq \img Q \oplus (\Ker P \cap \Ker Q)\]
if and only if
\[ 
 \img Q \leq \img P \oplus (\Ker P \cap \img Q) \oplus (\Ker P \cap \Ker Q) .
\]
\end{lemma}

We obtain the following characterization of the idempotency of $PQ$ in terms of the kernels and images of $P$ and~$Q$ alone. 

\begin{theorem} \label{thm:GTMainResult}
The following statements are equivalent:
\begin{enumerate} [label=(\roman*)]
 \item The composition $PQ$ is a projector. 		\label{PQProj1}
 \item $\img P \cap (\img Q + \Ker P) \leq \img Q \oplus ( \Ker P \cap \Ker Q)$ \label{PQProj2}
 \item $\img Q \leq \img P \oplus (\Ker P \cap \img Q) \oplus (\Ker P \cap \Ker Q)$ \label{PQProj4}
 \item $\Ker Q \oplus (\Ker P \cap \img Q) \geq \Ker P \cap (\img Q + \img P)$ \label{PQProj3}
 \item $\Ker P \geq \Ker Q \cap (\img Q + \Ker P) \cap (\img Q + \img P)$ \label{PQProj5}
\end{enumerate}
\end{theorem}

\begin{proof}
The equivalence of~\ref{PQProj1},~\ref{PQProj2},~\ref{PQProj4} follow from the previous lemma and Corollary~\ref{cor:KerImComp}. By the duality principle~\ref{DualPrin}, the last two conditions are equivalent to~\ref{PQProj2} and~\ref{PQProj4}, respectively. 
\end{proof}

For \agi s, it is also interesting to have sufficient conditions for $PQ$ as well as $QP$ to be projectors; 
for example, if $P$ and $Q$ commute. 
This can again be characterized in terms of the images and kernels of  $P$ and $Q$ alone.
If  $PQ = QP$, one sees with Corollary~\ref{cor:KerImComp}  that
\begin{equation} \label{eq:KerSum}
 \img PQ = \img P \cap \img Q \quad\text{and}\quad  \Ker PQ  = \Ker P + \Ker Q. 
\end{equation}
In general, these conditions are necessary 
but not sufficient for commutativity of $P$ and $Q$, see
\cite[Ex.\ 1]{GrossTrenkler1998}. 

Using Corollary~\ref{cor:KerImComp}, modularity~\eqref{eq:Modularity}, and~\eqref{eq:ModImplication}, one obtains the following characterization of projectors with image or kernel as in~\eqref{eq:KerSum}; for further details see~\cite{Korporal2012}.
For the commutativity of projectors see also~\cite[p. 339]{BrownPage1970}.

\begin{proposition} \label{prop:CorLem}
The composition $PQ$ is a projector with 
 \begin{enumerate} [label=(\roman*)]
 \item  $\img PQ = \img P \cap \img Q$ if and only if  
 \[\img Q = (\img P \cap \img Q) \oplus (\Ker P \cap \img Q).\]
 \item  $\Ker PQ = \Ker P + \Ker Q$ if and only if \[\Ker P = (\Ker P \cap \Ker Q) \oplus (\Ker P \cap \img Q).\]
 \end{enumerate}
\end{proposition}

\begin{corollary} \label{cor:PQ=QP}
  We have $PQ=QP$
if and only if
\begin{equation*}\label{eq:PQ=QPImQ}
 \img Q = (\img P \cap \img Q)  \oplus (\Ker P \cap \img Q)
\end{equation*}
 and
\begin{equation*}\label{eq:PQ=QPKerQ}
 \Ker Q =  (\img P \cap \Ker Q) \oplus (\Ker P \cap \Ker Q). 
\end{equation*}
\end{corollary}

In~\cite[Thm.~4]{GrossTrenkler1998} and~\cite[Thm.~3.2]{BaksalaryBaksalary2002} different necessary and sufficient conditions for the 
commutativity of two projectors are given, but both require the computation of $PQ$ as well as of $QP$. 

\section{Reverse order law for generalized inverses} \label{sec:ROL}

Proposition \ref{prop:CompIINS} and Theorem \ref{thm:GTMainResult} together give  necessary and
sufficient conditions for the reverse order law for outer inverses to hold, in terms of the 
defining spaces $B_i$ and $E_i$ alone.

\begin{theorem} \label{thm:G2G1OI}
 Let $T_1 \colon V \to W$ and $T_2 \colon U \to V$ be linear with \oi s $G_1=\OI(T_1, B_1, E_1)$ and $G_2=\OI(T_2, B_2, E_2)$. 
The following conditions are equivalent:
\begin{enumerate} [label=(\roman*)]
 \item $G_2G_1$ is an outer inverse of $T_1T_2$.		\label{G2G1OI}
 \item $T_2(B_2) \cap (B_1 + E_2) \leq B_1 \oplus ( E_2 \cap T_1^{-1}(E_1))$ \label{G2G1OI2}
 \item $B_1 \leq T_2(B_2) \oplus (E_2 \cap B_1) \oplus (E_2 \cap T_1^{-1}(E_1))$ \label{G2G1OI4}
 \item $T_1^{-1}(E_1) \oplus (E_2 \cap B_1) \geq E_2 \cap (B_1 + T_2(B_2))$ \label{G2G1OI3}
 \item $E_2 \geq T_1^{-1}(E_1) \cap (B_1 + E_2) \cap (B_1 + T_2(B_2))$ \label{G2G1OI5}
\end{enumerate}
\end{theorem}

\begin{proof} 
Recall that $\img G_i = B_i$ and $\Ker G_i = E_i$, and $Q=T_2G_2$ and $P = G_1T_1$ are projectors
with
\begin{equation*}
 \img P = B_1, \quad \Ker P = T_1^{-1}(E_1), \quad \img Q = T_2(B_2), \quad \text{and} \quad
 \Ker Q = E_2.
\end{equation*}
By Proposition \ref{prop:CompIINS}, $G_2G_1$ is an outer inverse if and only if $QP$ is a projector.
 Applying Theorem \ref{thm:GTMainResult} proves the claim.
\end{proof}

In the following theorem, we give the analogous
conditions for inner inverses, where $P =G_1T_1$  and $Q = T_2G_2$ 
are the projectors corresponding to the direct sums in \eqref{eq:DirSumDec}.
Note that the conditions for inner inverses only depend on the choice of $B_1$ and $E_2$,
but not on  $B_2$ and $E_1$.

The characterization \ref{G2G1GI4} and the orthogonal of  \ref{G2G1GI5} in the following theorem generalize 
\cite[Thm. 2.3]{Werner1992} to arbitrary vector spaces.

\begin{theorem}\label{thm:G2G1GI}
 Let $T_1 \colon V \to W$ and $T_2 \colon U \to V$ be linear with inner inverses $G_1 \in \II(T_1, B_1, E_1)$ 
 and $G_2 \in \II(T_2, B_2, E_2)$.
The following conditions are equivalent:
\begin{enumerate} [label=(\roman*)]
 \item $G_2G_1 $ is an inner inverse of $T_1T_2$. 		\label{G2G1GI}
 \item $B_1 \cap (\img T_2 + \Ker T_1) \leq \img T_2 \oplus ( \Ker T_1 \cap E_2)$ \label{G2G1GI2}
 \item $\img T_2 \leq B_1 \oplus (\Ker T_1 \cap \img T_2) \oplus (\Ker T_1 \cap E_2)$ \label{G2G1GI4}
 \item $E_2 \oplus (\Ker T_1 \cap \img T_2) \geq \Ker T_1 \cap (\img T_2 + B_1)$ \label{G2G1GI3}
 \item $\Ker T_1 \geq E_2 \cap (\img T_2 + \Ker T_1) \cap (\img T_2 + B_1)$ \label{G2G1GI5}
\end{enumerate}
\end{theorem}

The question when the reverse order law holds for all inner inverses of $T_1$ and $T_2$ was answered for matrices
in \cite[Thm. 2.3]{Werner1994}, and an alternative proof was given in  \cite{Gross1997}. 
Using the previous characterizations, we give a short proof that generalizes the result to arbitrary vector spaces.

\begin{theorem}\label{thm:allinner}
 Let $T_1 \colon V \to W$ and $T_2 \colon U \to V$ be linear. Then $G_2G_1$ is an inner inverse of $T_1T_2$ for all
 inner inverses $G_1$ of $T_1$ and $G_2$ of $T_2$ if and only if $T_1T_2=0$ or $\Ker T_1 \leq \img T_2$.
\end{theorem}
 
\begin{proof}
 If  $\Ker T_1 \leq \img T_2$ then $\Ker T_1 \cap \img T_2=\Ker T_1$ and~\ref{G2G1GI4} in the previous theorem is satisfied since $\Ker T_1+B_1=V$. The case $T_1T_2=0$ is trivial.

For the reverse implication, assume that  $\img T_2$ is not contained in $\Ker T_1$ and $\Ker T_1$ is not contained in $\img T_2$.  Choose $V_1, V_2 \leq V$  such that we have two direct sums
   $\Ker T_1 =(\img T_2 \cap \Ker T_1) \oplus V_1 $ and $ \img T_2 = (\img T_2 \cap \Ker T_1) \oplus V_2 $.
 Then we have  
    \begin{equation}
  \label{eq:DirImKer}
  \img T_2 + \Ker T_ 1= (\img T_2 \cap \Ker T_1) \oplus V_1 \oplus V_2.
 \end{equation}
By assumption, we can choose non-zero $v_1\in V_1$ and $v_2 \in V_2$. Let $v=v_1+v_2$. Then $v\in  \img T_2 + \Ker T_1$ and $v \not \in \Ker T_1$, $v \not \in \img T_2$.  Hence we can choose $B_1$ and $E_2$ such that $v\in B_1$ and $v \in E_2$ and $V= \Ker T_1 \oplus B_1=\img T_2 \oplus E_2$. 
 Then
 \[
 v\in E_2 \cap (\img T_2 + \Ker T_1) \cap (\img T_2 + B_1)
 \]
but $v\in \Ker T_1$. Hence \ref{G2G1GI5} in the previous theorem is not satisfied for inner inverses with $\img G_1 = B_1$ and  $\Ker G_2 = E_2$. 
\end{proof}

Werner \cite[Thm. 3.1]{Werner1992} proves that for matrices it is always possible to construct inner
inverses such that the reverse order law holds. Using the necessary and sufficient condition for outer inverses above, 
we extend this result to \agi s in arbitrary vector spaces. The special case of Moore-Penrose inverses is treated in \cite[Thm. 3.2]{ShinozakiSibuya1974}, 
and explicit solutions are constructed in \cite{ShinozakiSibuya1979, WibkerHoweGilbert1979}.

\begin{theorem} \label{prop:Construct}
 Let $T_1 \colon V \to W$ and $T_2 \colon U \to V$ be linear.  There always exist algebraic generalized inverses $G_1$ of $T_1$ and $G_2$ of $T_2$ such that $G_2 G_1$ is an algebraic generalized inverse of $T_1T_2$.
\end{theorem}

\begin{proof}
 Choose $V_1, V_2 \leq V$ as in the previous proof such that~\eqref{eq:DirImKer} holds.  Moreover, choose $V_3 \leq V$  such that
 \begin{equation*}
  V = (\img T_2 + \Ker T_1) \oplus V_3 = (\img T_2 \cap \Ker T_1) \oplus V_1 \oplus V_2 \oplus V_3. 
 \end{equation*}
 Then $B_1=V_2 \oplus V_3$ is a direct complement of $\Ker T_1$, and $E_2=V_1 \oplus V_3$ is a direct complement of $\img T_2$.
 Hence there exist respectively an \agi\ $G_1$ of $T_1$ with $\img G_1 = B_1$ and  $G_2$ of $T_2$ with
 $\Ker G_2 = E_2$. We verify that such $G_1$ and $G_2$ 
  satisfy Theorem~\ref{thm:G2G1OI}~\ref{G2G1OI4}, where $T_1^{-1}(E_1) =\Ker T_1$ and $ T_2(B_2)= \img T_2$ since $G_1$ and $G_2$ are \agi s:
 \begin{equation*}
  \img T_2 \oplus (E_2 \cap B_1) \geq \img T_2 \oplus V_3 = (\img T_2 \cap \Ker T_1) \oplus V_2 \oplus V_3 \geq B_1.
 \end{equation*}
 Similarly, we verify Theorem~\ref{thm:G2G1GI}~\ref{G2G1GI4}
 \begin{equation*}
  B_1  \oplus (\Ker T_1 \cap \img T_2) = V_2 \oplus V_3 \oplus (\Ker T_1 \cap \img T_2) \geq 
  V_2  \oplus (\Ker T_1 \cap \img T_2) = \img T_2. 
   \end{equation*}
  Hence $G_2G_1$ is an \agi\ of $T_1T_2$ for all $G_1=\GI(T_1,B_1,E_1)$ and $G_2=\GI(T_2,B_2,E_2)$, 
  independent of the choice of $E_1$ and $B_2$.
\end{proof}

\section{Representing the product of outer inverses}

In this section, we assume that for two linear maps $T_1 \colon V \to W$ and $T_2 \colon U \to V$ 
with outer inverses $G_1$ and $G_2$ the reverse order law holds. Our goal is to find a description of the product $G_2G_1$ that does not require the explicit knowledge of $G_1$ and $G_2$.  Using the representation via projectors, one immediately verifies that
\begin{equation*} 
 \OI(T_2, P_2, Q_2)\OI(T_1, P_1, Q_1)= \OI(T_1T_2, P_2 - G_2P_1T_2, T_1Q_2G_1)
\end{equation*}
but this expression involves both outer inverses $G_1$ and $G_2$. For the representation via defining spaces, we
compute the kernel and the image of the product.

\begin{lemma} \label{lem:KerImG2G1}
 Let $T_1 \colon V \to W$ and $T_2 \colon U \to V$ be linear with outer inverses $G_1 = \OI(T_1, B_1, E_1)$ and $G_2= \OI(T_2, B_2, E_2)$.
Then
\begin{equation*}
 \Ker G_2G_1 =  \es_1 \oplus T_1(B_1 \cap \es_2) \quad \text{and} \quad \img G_2G_1 = G_2 ((B_1 + E_2) \cap \img T_2).
\end{equation*}
\end{lemma}

\begin{proof}
Recall that by definition $\Ker G_i = E_i$ and $\img G_i = B_i$. The first identity follows directly from Proposition \ref{prop:KerComp}.  For the second identity, we first note that for a linear map $G$ and subspaces $V_1,V_2$, we have 
 $G(V_1 \cap V_2) = G(V_1) \cap G(V_2)$ if $\Ker G \leq V_1$. Hence $ G_2 ((B_1 + E_2) \cap \img T_2)$ equals
 \[
   G_2((\img G_1 + \Ker G_2)  \cap \img T_2) = G_2(\img G_1) \cap G_2(\img T_2) = \img G_2G_1,
 \]
since $G_2(\img T_2) = \img G_2$ by Proposition~\ref{prop:OIChar}~\ref{oi:Nashed2}.
\end{proof}

Note that the expression for the image of the composition requires the explicit knowledge of $G_2$. 
In particular, the reverse order law takes the form 
\begin{equation*} 
  \OI(T_2, B_2, E_2)\OI(T_1, B_1, E_1) = \OI(T_1T_2, G_2((B_1 + E_2)\cap \img T_2), E_1 +T_1(B_1 \cap E_2)).
\end{equation*}
Werner \cite[Thm. 2.4]{Werner1992} gives a result in a similar spirit for inner inverses of matrices.

Using an implicit description of $\img G_i$, it is possible to state the reverse order law in a form that depends on the kernels and images of the respective outer inverses alone.
This approach is  motivated by our application to linear 
boundary problems (Section~\ref{sec:BP}), where it is natural to define solution spaces 
via the boundary conditions they satisfy.

In more detail, the Galois connection from Appendix \ref{sec:Duality} allows to represent 
a subspace $B$ implicitly via the orthogonally closed subspace $\B = B^{\perp}$ 
of the dual space.  
We will therefore use the notation $G=\OI(T, \B, E)$ 
for the outer inverse with $\img G = \Bo$ and $\Ker G = E$ as well as 
the analogue for inner inverses. 

\begin{theorem} \label{thm:RevOrderLaw}
 Let $T_1 \colon V \to W$ and $T_2 \colon U \to V$ be linear with \oi s $G_1= \OI(T_1, \B_1, E_1)$ and 
$ G_2 = \OI(T_2, \B_2, E_2)$.
 If $G_2G_1$ is an \oi\ of $T_1T_2$,  then
 \begin{equation} \label{eq:DefComp}
  \OI(T_2, \B_2, E_2)\OI(T_1, \B_1, E_1)= \OI(T_1T_2, \B_2 \oplus T_2^*(\B_1 \cap \es_2^{\perp}), \es_1 \oplus T_1(\Boi{1} \cap \es_2)),
 \end{equation}
 where $T_2^*$ denotes the transpose of $T_2$.
\end{theorem}

\begin{proof}
From Lemma \ref{lem:KerImG2G1} we 
already know that
  $ \Ker G_2G_1 
= \es_1 \oplus T_1(\Boi{1} \cap \es_2)$.
 From Proposition \ref{prop:Duals} and \ref{prop:KerComp} we get 
\begin{multline*} 
 (\img G_2G_1)^{\perp}
=  \Ker G_1^*G_2^*
= T_2^*(\Ker G_1^* \cap \img G_2^*) \oplus \Ker G_2^*\\
= T_2^*((\img G_1)^{\perp} \cap (\Ker G_2)^{\perp} ) \oplus (\img G_2)^{\perp}
= T_2^*(\B_1 \cap E_2^{\perp}) \oplus \B_2,
\end{multline*}
and thus \eqref{eq:DefComp} holds.
\end{proof}

A computational advantage of this representation is that one can determine $G_2G_1$ directly by  computing
only one outer inverse instead of computing both $G_1$ and $G_2$; see the next section for an example. 

\section{Examples for matrices} \label{sec:Example}

In this section, we illustrate our results for finite-dimensional vector spaces. 
In particular, we show how to compute directly the composition of two generalized inverses using the reverse order 
law in the form~\eqref{eq:DefComp}.

Consider the following linear maps $T_1 \colon \Q^4 \to \Q^3$ and $T_2 \colon \Q^3 \to \Q^4$ given by
\begin{equation*}
 T_1 = 
\begin{pmatrix}
 1 & -1 & -1 & 1\\
 0 &  2 &  2 & -2\\
 3 &  1 &  1 & -1\\
\end{pmatrix}
\quad \text{and} \quad
T_2 = 
\begin{pmatrix}
 1 & -2 & -1 \\
 1 &  1 &  2 \\
-1 &  5 &  4 \\
-1 &  5 &  4 \\
\end{pmatrix}.
\end{equation*}
  We first use Theorem \ref{thm:G2G1OI} and \ref{thm:G2G1GI} to check 
whether for  \agi s $G_1 = \GI(T_1, B_1, E_1)$ and $G_2=\GI(T_2, B_2, E_2)$ the composition 
$G_2G_1$ is an \agi\ of $T_1T_2$.

For testing the conditions, we only need to fix 
$B_1 = \img G_1$ 
and $E_2 = \Ker G_2$,
such that $B_1 \oplus \Ker T_1 = \Q^4 = E_2 \oplus \img T_2$. 
We have
\begin{equation*}
 \Ker T_1 = \Ll((0,1,0,1)^T, (0,0,1,1)^T),\quad  \img T_2 = \Ll( (1,0,-2,-2)^T, (0,1,1,1)^T),
\end{equation*}
so we may choose for example
\begin{equation*}
  B_1 = \Ll((1,0,0,0)^T, (0,1,0,0)^T), \quad
E_2 = \Ll((1,0,0,0)^T, (0,0,1,0)^T).
\end{equation*}
For \agi s,  we obtain as a necessary and sufficient condition for being an outer inverse
\begin{equation*}
 B_1 \leq \img T_2  \oplus (E_2 \cap B_1) \oplus (E_2 \cap \Ker T_1)
\end{equation*}
from Theorem \ref{thm:G2G1OI} \ref{G2G1OI4}.

Since
$E_2 \cap \Ker T_1= \{0\}$ and
 $E_2 \cap B_1 = \Ll((1,0,0,0)^T)$, the right hand side yields that
$\Ll ( (1,0,0,0)^T, (0,1,0,0)^T, (0,0,1,1)^T ) \geq B_1.$
Thus for all \agi s $G_1$ and $G_2$ with $\img G_1= B_1$ and $\Ker G_2 = E_2$, the product $G_2G_1$ is an outer inverse of
$T_1T_2$. 

The corresponding condition for inner inverses by Theorem \ref{thm:G2G1GI} \ref{G2G1GI4} is
\begin{equation*}
\img T_2 \leq B_1 \oplus (\Ker T_1 \cap \img T_2) \oplus (\Ker T_1 \cap E_2). 
\end{equation*}
Since $\Ker T_1 \cap \img T_2 = \{0\}$, the right hand side yields $B_1$, which does not contain $\img T_2$.
Hence for the above choices of $G_1$ and $G_2$, the product $G_2G_1$ is never an inner inverse of $T_1T_2$.

Since $G_2G_1$ is an outer inverse, Theorem \ref{thm:RevOrderLaw} allows to determine $G_2G_1$ directly 
without knowing the factors. 
Identifying the dual space with row vectors, the orthogonals of $B_1$ and $E_2$ are given by 
\begin{equation*}
 B_1^{\perp} = \B_1 = \Ll((0,0,1,0), (0,0,0,1)), \quad E_2^{\perp}= \Ll((0,1,0,0), (0,0,0,1)),
\end{equation*}
so we have $\Bo_1 \cap E_2 =  \Ll((1,0,0,0)^T)$ and $\B_1 \cap E_2^{\perp} =\Ll((0,0,0,1))$. For explicitly
computing $G_2G_1$, we also have to choose $B_2 = \img G_2$ and $E_1 = \Ker G_1$.
Since we have
\begin{equation*}
 \img T_1 = \Ll((1,0,3)^T), (0,1,2)^T), \quad 
 \Ker T_2 =\Ll((1,1,-1)^T),
\end{equation*}
we may choose the complements $E_1 = \Ker G_1$ and $B_2= \img G_2$ as 
\begin{equation*}
 E_1 = \Ll((0,0,1)^T) \quad \text{and} \quad 
 B_2 = \Ll((1,0,0)^T, (0,1,0)^T). 
\end{equation*}
Using \eqref{eq:DefComp}, we can determine the kernel
\begin{equation*}
 E = \Ker G_2G_1 = E_1 \oplus T_1(\Bo_1 \cap E_2) = \Ll((1,0,0)^T, (0,0,1)^T).
\end{equation*}
The image of $G_2G_1$ is by \eqref{eq:DefComp} given via the orthogonal 
\begin{equation*}
 (\img G_2G_1)^{\perp} = \B_2 \oplus T_2^*(\B_1 \cap E_2^{\perp})
= \Ll((0,0,1), (-1, 5, 4)),
\end{equation*}
which means that $B = \img G_2G_1 = \Ll((5,1,0)^T)$. Therefore we can directly determine $G$ as the unique
outer inverse 
\begin{equation*}
 G = \OI(T_1T_2, B, E) = \begin{pmatrix}
      0 & \frac{5}{12} &0 \\ 0 &  \frac{1}{12} &0 \\0&0&0
     \end{pmatrix}.
\end{equation*}
One easily checks that $G$ is an outer inverse of $T$. 

\section{Fredholm operators} \label{sec:Fredholm}

We now turn to algorithmic aspects of the previous results. As already emphasized, for arbitrary vector spaces we
can express conditions for the reverse order law in terms of the defining spaces alone.
Nevertheless, in general it will not be possible to compute sums and intersections of infinite-dimensional subspaces.
For algorithmically checking the conditions of Theorem \ref{thm:G2G1OI} or \ref{thm:G2G1GI} and for computing the reverse
order law in the form \eqref{eq:DefComp}, we 
consider finite (co)dimensional spaces and Fredholm operators.

Recall that a linear map $T$ between vector spaces is called  \emph{Fredholm} operator if $\dim \Ker T < \infty$ and
 $\codim \img T < \infty$. Moreover, for finite codimensional subspaces $V_1\leq V$, we have  $\codim V_1=\dim V_1^{\perp}$. 
 In this case, $V_1$ can be implicitly represented by the finite-dimensional subspace
$V_1^{\perp} \leq V^*$. For an application to linear ordinary boundary problems, see the next section.
 
We assume that for finite-dimensional subspaces, we can compute sums and intersections and check inclusions, 
both in vector spaces and in their duals. With the following lemma, the intersection of a finite-dimensional subspace with a finite codimensional subspace is reduced to computing kernels of matrices. 

\begin{definition}
Let $u=(u_1, \ldots, u_m)^T \in V^m$ and $\beta = (\beta_1, \ldots, \beta_n)^T \in (V^*)^n$.
We call
\begin{equation*}
  \beta(u) = \begin{pmatrix}
              \beta_1(u_1) & \ldots & \beta_1(u_m)\\
               \vdots      & \ddots & \vdots \\
	      \beta_n(u_1) & \ldots & \beta_n(u_m)
             \end{pmatrix} \in F^{n \times m}
 \end{equation*}
 the \emph{evaluation matrix} of $\beta$ and $u$. 
\end{definition}

\begin{lemma} \label{lem:intersections}
 Let $U \leq V$ and $\B \leq V^*$ be generated respectively by $u=(u_1, \ldots, u_m)$ and $\beta=(\beta_1, \ldots, \beta_n)$.
 Let $k^1, \ldots, k^r \in F^m$ be a basis of $\Ker \beta(u)$, and $\kappa^1, \ldots, \kappa^s \in F^n$ a basis 
 of $\Ker (\beta(u))^T$.   Then 
 \begin{enumerate}[label=(\roman*)]
\item $U\cap \Bo$ is generated by
   $\sum_{i=1}^m k^1_i u_i , \ldots, \sum_{i=1}^m k^r_i u_i$ and
 \item $U^{\perp} \cap \B $ is generated by
  $\sum_{i=1}^n \kappa^1_i\beta_i , \ldots, \sum_{i=1}^n \kappa^s_i \beta_i$.
 \end{enumerate}
\end{lemma}

\begin{proof}
 A linear combination $v=\sum_{\ell = 1}^m c_{\ell}u_{\ell}$ is in $\Bo$ if and only if 
 $\beta_i(v)= 0$ for $1 \leq i \leq n$,  that is,  $\sum_{\ell = 1}^m c_{\ell} \beta_i(u_{\ell}) = 0$ for  $1 \leq i \leq n$. Hence $\beta(u) \cdot (c_1, \ldots, c_m)^T = 0$. 
Analogously, one sees that the coefficients of linear combination in $U^{\perp} \cap \B$ are in the kernel of $(\beta(u))^T$.
\end{proof}

We reformulate the conditions of Theorem \ref{thm:G2G1OI} such that
for Fredholm operators they only involve 
operations on finite-dimensional subspaces and intersections
like in the previous lemma. Similarly, one can rewrite the conditions of Theorem \ref{thm:G2G1GI}.

\begin{corollary} \label{cor:OIF}
 Let $T_1 \colon V \to W$ and $T_2 \colon U \to V$ be linear with \oi s $G_1=\OI(T_1, \B_1, E_1)$ and $G_2=\OI(T_2, \B_2, E_2)$. Let $\C_2=T_2(\Bo_2)^\perp$ and $K_1=T_1^{-1}(E_1)$.  
The following conditions are equivalent:
\begin{enumerate} [label=(\roman*)]
 \item $G_2G_1$ is an outer inverse of $T_1T_2$. 	\label{algOI1}	
 \item $\C_2 + (\B_1 \cap E_2^{\perp}) \geq \B_1 \cap (E_2 \cap K_1)^{\perp}$ \label{algOI2}	
 \item $\B_1 \geq \C_2 \cap (E_2 \cap \Bo_1)^{\perp} \cap (E_2 \cap K_1)^{\perp}$ \label{algOI3}	
 \item $K_1  \oplus (E_2 \cap \Bo_1) \geq E_2 \cap (\B_1 \cap \C_2)^{\perp}$ \label{algOI4}	
 \item $E_2 \geq K_1 \cap (\B_1 \cap E_2^{\perp})^{\perp} \cap (\B_1 \cap \C_2)^{\perp}$ \label{algOI5}	
\end{enumerate}
\end{corollary}

\begin{proof} 
 Taking the orthogonal of both sides of \ref{thm:G2G1OI} \ref{G2G1OI2}, \ref{G2G1OI4} respectively
and applying Proposition \ref{prop:SumCap} we get \ref{algOI2} and \ref{algOI3}.
 For \ref{algOI4} and \ref{algOI5}, we can  apply Proposition \ref{prop:SumCap} directly to the corresponding 
conditions of Theorem \ref{thm:G2G1OI}.
\end{proof}

We note that using Lemma \ref{lem:intersections}, it also possible to 
determine constructively the implicit representation \eqref{eq:DefComp} of a 
product of generalized inverses; see the next section.

\section{Examples for linear ordinary boundary problems} \label{sec:BP}

As an example involving infinite dimensional spaces and Fredholm operators, we consider solution (Green's) operators for linear ordinary boundary problems.
Algebraically, linear boundary problems can be represented as a pair $(T, \B)$, where $T \colon V \to W$ is a surjective linear map, 
and $\B \leq V^*$ is an orthogonally closed subspace of (homogeneous) boundary conditions. We say that $v\in V$ is a solution of $(T, \B)$ for a given $w\in W$ if $Tv=w$ and $v\in \Bo$. 

For a regular boundary problem (having a unique solution for every right-hand side), the Green's
operator is defined as the unique right inverse $G$ of $T$ with $\img G = \Bo$; see \cite{RegensburgerRosenkranz2009}
for further details.
The product $G_2G_1$ of the Green's operators  of two boundary problems $(T_1, \B_1)$ and $(T_2, \B_2)$
is then the Green's operator of the regular boundary problem $(T_1T_2, \B_2 \oplus T_2^*(\B_1))$, see also Theorem~\ref{thm:RevOrderLaw}.

For boundary problems having at most one solution, that is $\Bo \cap \Ker T = \{0\}$, the linear algebraic setting has been extended in~\cite{Korporal2012} by defining generalized Green's operators as outer inverses. More specifically, one first has to project an arbitrary right-hand side $w\in W$ onto $T(\Bo)$, the image of the ``functions'' satisfying the boundary conditions, along a complement $E$ of $T(\Bo)$. 
The corresponding generalized Green's operator is defined as  the outer inverse $G=\OI(T, \B, E)$, and we refer to  $E \leq W$ as an \emph{exceptional space} for the boundary problem $(T,\B)$. 

The question when the product of two outer inverses is again an outer inverse, is the basis for factoring boundary problems
into lower order problems; see~\cite{RegensburgerRosenkranz2009,RosenkranzRegensburger2008a} for the case of regular boundary problems. 
This, in turn, provides a method to factor certain integral operators.

As an example, let us consider the boundary problem
\begin{equation} \label{eq:StandardEx}
 \bvp{u''=f}{u'(0)=u'(1)=u(1)=0.}
\end{equation}
In the above setting, this means we consider the pair $(T_1,\B_1)$ with $T_1=\D^2$ and  $\B_1 = \Ll(\E_0\D, \E_1\D,\E_1)$, 
where $\D$ denotes the usual derivation on smooth functions and $\E_c$ the evaluation at $c \in \R$. 
The boundary problem is only solvable for \emph{forcing functions} $f$ satisfying the \emph{compatibility condition} 
$\smallint_0^1 f(\xi) \dx\xi  =0$; more abstractly, we have $T_1(\Bo_1)=\C_1^\perp$ with 
$\C_1=\Ll(\smallint_0^1)$, where $\smallint_0^1$ denotes the functional $f\mapsto \smallint_0^1 f(\xi)\dx\xi $.  
For computing a generalized Green's operator of $(T_1, \B_1, E_1)$, we have to project $f$ onto $\C_1^\perp$  
along a fixed complement $E_1$. In~\cite{KorporalRegensburgerRosenkranz2011}, we   
computed the generalized Green's operator
\begin{equation*}
 G_1(f) =  x \smallint_0^x f(\xi) \dx \xi - \smallint_0^x \xi f(\xi)\dx \xi - \frac{1}{2}(x^2+1) \smallint_0^1 f(\xi)\dx \xi 
 + \smallint_0^1 \xi f(\xi)\dx\xi
\end{equation*}
of \eqref{eq:StandardEx} for $E_1= \R$ being the constant functions. It is easy to see that in this case we have
$T_1^{-1}(E_1) = \Ll(1, x, x^2)$.

As a second boundary problem, we consider
\begin{equation*} 
 \bvp{u''-u=f}{u'(0)=u'(1)=u(1)=0,}
\end{equation*}
or $(T_2,\B_2)$ with $T_2=\D^2-1$ and $\B_2=\Ll(\E_0\D, \E_1\D, \E_1)$.
For the corresponding generalized Green's operator $G_2$  with exceptional space $E_2= \Ll(x)$, we will now check 
if  the products $G_1G_2$ and $G_2G_1$ are again generalized Green's operators of $T_1T_2=T_2T_1=\D^4-\D^2$, 
using condition~\ref{algOI2} of
Corollary~\ref{cor:OIF}.

We use the algorithm from~\cite{KorporalRegensburgerRosenkranz2011}, implemented in the package \pack\ for 
the computer algebra system \maple, 
to compute the compatibility conditions.  The algorithm is based on the identity
\[
T(\Bo)^\perp=G^*(\B \cap (\Ker T)^\perp),
\]
for any right inverse $G$ of $T$, which follows from Propositions~\ref{prop:Duals} and~\ref{prop:KerComp}. Moreover, a right inverse of the differential operator can be computed by variation of constants and the intersection $\B \cap (\Ker T)^\perp$ using Lemma~\ref{lem:intersections}. Thus we obtain 
 $\C_2 = \Ll (\smallint_0^1 (\exp(-x) + \exp(x)) )$, where $\smallint_0^1 (\exp(-x) + \exp(x))$ denotes  
 the functional $f\mapsto \smallint_0^1 (\exp(-\xi) +\exp(\xi)) f(\xi) \dx \xi$.

The space $T_2^{-1}(E_2)= \Ll(x, \exp(x), \exp(-x))$
can be computed using Proposition~\ref{prop:KerComp} and a right inverse of the differential operator;  
this is also implemented in the \pack\ package.
Hence we have $E_1 \cap T_2^{-1}(E_2) = \{0\}$ and therefore
$\B_2 \cap(E_1 \cap T_2^{-1}(E_2))^{\perp} = \B_2$. Computing $\B_2 \cap E_1^{\perp}$ with Lemma~\ref{lem:intersections}
yields $\B_2 \cap E_1^{\perp} = \Ll(\E_0\D, \E_1\D)$; thus $G_1G_2$ is not an outer inverse of the product $T_2T_1=\D^4-\D^2$ 
by Corollary~\ref{cor:OIF}~\ref{algOI2}.

On the other hand, we have $E_2 \cap T_1^{-1}(E_1) = \Ll(x) = E_2$, hence we know by Corollary~\ref{cor:OIF}~\ref{algOI2} that 
$G_2G_1$ is an outer inverse of $T_1T_2 = \D^4-\D^2$.
Furthermore, by Theorem~\ref{thm:RevOrderLaw} we can determine which boundary problem it solves without 
computing  $G_1$ and $G_2$. With Lemma~\ref{lem:intersections} we obtain $\Bo_1 \cap E_2=\{0\}$ and 
 $\B_1 \cap E_2^{\perp}= \Ll(\E_0\D- \E_1, \E_1\D-\E_1)$.
Since applying the transpose $T_2^*$ to $\E_0\D- \E_1$ and $ \E_1\D-\E_1$ corresponds to multiplying $T_2=\D^2-1$ from the right, 
$G_2G_1$ is the generalized Green's operator of
\begin{equation*}
 (\D^4-\D^2, \Ll ( \E_0\D, \E_1\D, \E_1, \E_0\D^3 - \E_1\D^2, \E_1\D^3 - \E_1\D^2  ),  \R )
\end{equation*}
by \eqref{eq:DefComp}; or, in traditional notation, it solves the boundary problem
\begin{equation*} 
 \bvp{u''''-u''=f}{u'(0)=u'(1)=u(1)=u'''(0)-u''(1)=u'''(1)-u''(1)=0,}
\end{equation*}
with exceptional space $\R$.
\section*{Acknowledgements}

We would like to thank an anonymous referee for his detailed comments. G.R.~was supported by the Austrian Science Fund (FWF): J 3030-N18.

\appendix

\section{Duality} \label{sec:Duality}
In the appendix, we summarize duality results for arbitrary vector spaces and their duals that
generalize the standard duality for finite-dimensional vector spaces but do not require any topological assumptions;
see \cite[Section 9.2 and 9.3]{Kothe1969} and \cite{RegensburgerRosenkranz2009} for further 
details. 
The notation should also remind of the analogous and well-known results for Hilbert spaces.

Let $V$ and $W$ be vector spaces over a field $F$ and $\scal{}{}\colon V \times W
\rightarrow F$ be a bilinear map. 
For $V_1 \leq V$, we define the orthogonal
\begin{equation*}
 V_1^{\perp} = \{w \in W \mid \scal{v}{w} = 0 \text{ for all } v \in V_1\} \leq W.
\end{equation*}
The orthogonal $W_1^{\perp}$ for $W_1 \leq W$ is defined analogously.
A subspace $U$ is called orthogonally closed if $U = U^{\perp\perp}$.
 It follows directly from the definition that for all subsets $X_1, X_2 \subseteq V$, we have
 $X_1 \subseteq X_2 \Rightarrow X_1^{\perp} \supseteq  X_2^{\perp}$ and  $X_1 \subseteq X_1^{\perp\perp}$;
and the same holds for subsets of $W$. Let $\PG(V)$ denote the projective geometry of $V$, that is, the 
 partially ordered set (poset) of all subspaces ordered by inclusion. 
 Then we have an order-reversing Galois connection
between $\PG(V)$ and $\PG(W)$ defined by $U \mapsto U^{\perp}$.
 
We now consider the canonical bilinear form $V \times V^* \to F$ of a vector space $V$ and its dual $V^*$
defined by $\scal{v}{\beta} \mapsto \beta(v)$. Then every subspace $V_1 \leq V$ is orthogonally closed 
with respect to the canonical bilinear form, 
and every finite-dimensional subspace $\B \leq V^*$ is orthogonally
closed. 
The Galois connection
gives an order-reversing bijection between $\PG(V)$ and
the poset of all orthogonally closed subspaces of $V^*$. So we can describe any
subspace $V_1\leq V$ implicitly by the corresponding orthogonally closed
subspace $V_1^{\perp}$. 
We denote the poset of all orthogonally closed subspaces of $V^*$ with $\overline{\PG}(V^*)$.

The projective geometry $\PG(V)$ is a modular lattice, where join and meet are defined as the sum and intersection of
subspaces. Modularity means that for all $V_1, V_2, V_3 \in \PG(V)$ with $V_1 \leq V_3$ we have
\begin{equation} \label{eq:Modularity}
 V_1 + (V_2 \cap V_3) = (V_1 + V_2) \cap V_3.
\end{equation}
Moreover, for spaces $V_1 \leq V_3$ and $V_2 \leq V_4$, we have
\begin{equation} \label{eq:ModImplication}
 V = V_1 + V_2 = V_3 \oplus V_4 \; \Rightarrow \; V_1 = V_3\; \text{and}\; V_2=V_4,
\end{equation}
since $V_3 \cap V_4 = \{0\}$ implies $V_3 = (V_1 \oplus V_2) \cap V_3  = V_1$ and 
$V_4 = (V_1 \oplus V_2) \cap V_4 = V_2$.

One can also show that  $\overline{\PG}(V^*)$ is a modular lattice, where the meet is the intersection and the join is the orthogonal closure
of the sum of subspaces. Using this fact, one
can prove in particular that the sum of two orthogonally closed subspaces is orthogonally closed.  
The following theorem summarizes Section 9.3 of~\cite{Kothe1969}.
\begin{proposition} \label{prop:SumCap}
 The map $V_1 \mapsto V_1^{\perp}$ gives an order-reversing lattice isomorphism with inverse $\B_1 \mapsto \Bo_1$ between the
 complemented modular lattices $\PG(V)$ and $\overline{\PG}(V^*)$. In particular, 
the intersection of orthogonally closed subspaces in $V^*$ is orthogonally closed and 
   \begin{equation*}
     (V_1 + V_2)^{\perp} = V_1^{\perp} \cap V_2^{\perp} \quad \text{and} \quad 
     (\B_1 \cap \B_2)^{\perp} = \B_1^{\perp} + \B_2^{\perp}.
   \end{equation*}
   The sum of two orthogonally closed subspaces in $V^*$ is orthogonally closed and
   \begin{equation*}
    (V_1 \cap V_2)^{\perp} = V_1^{\perp} + V_2^{\perp} \quad \text{and} \quad 
    (\B_1 + \B_2)^{\perp} = \B_1^{\perp} \cap \B_2^{\perp}.
   \end{equation*}
  Furthermore, orthogonality preserves direct sums, such that
\begin{equation*} 
 V = V_1 \oplus V_2 \; \Rightarrow \; V^* = V_1^{\perp} \oplus V_2^{\perp}\quad\text{and}\quad V^* = \B_1 \oplus \B_2  \; \Rightarrow \; V = \B_1^{\perp} \oplus \B_2^{\perp}.
\end{equation*}
\end{proposition}

For a linear map $A \colon V \to W$ between vector spaces, the \emph{transpose} $A^* \colon  W^* \to V^*$ is defined by 
$\gamma \mapsto \gamma \circ A$. The transposition map $A \mapsto A^*$ from $L(V, W)$ to $L(W^*,V^*)$ is linear, and it is injective since
for all $ w \not= 0$ there exists a linear map $h \in W^*$ with $h(w) \not = 0$. 
Moreover, the transpose of a composition is given by $(A_1A_2)^* = A_2^*A_1^*$. 

The image of an orthogonally closed space under the transpose map is orthogonally closed,
and we have following identities, see, for example, \cite[Prop.~A.6]{RegensburgerRosenkranz2009}.

\begin{proposition}  \label{prop:Duals}
 Let $V$ and $W$ be vector spaces and $A \colon V \to W$ be linear. Then
 \begin{align*}
  A(V_1)^{\perp} &= (A^*)^{-1}(V_1^{\perp}), \quad  &A(\Boi{1}) &= (A^*)^{-1}(\B_1)^{\perp}, \\
  A^*(\C_1)^{\perp} &= A^{-1}(\C_1^{\perp}), \quad &A^*(W_1^{\perp}) &= A^{-1}(W_1)^{\perp},
 \end{align*}
 for subspaces $V_1 \leq V$, $W_1 \leq W$, $\C_1 \leq W^*$ and orthogonally closed subspaces $\B_1 \leq V^*$. In particular,
\begin{equation*}
  (\img A) ^{\perp} = \Ker A^*, \quad \img A =(\Ker A^*)^{\perp}, \quad
  (\img A^*)^{\perp} = \Ker A, \quad \img A^* = (\Ker A)^{\perp}, 
 \end{equation*} 
 for the image and kernel of $A$ and $A^*$.
\end{proposition}

The property of being a projector, outer/inner/\agi\  carries over to
the transpose.

\begin{proposition} \label{prop:PP*proj}
 A linear map $P \colon V \to V$ is a projector if and only if its transpose $P^*$ is a projector. 
 A linear map $G \colon W \to V$ is an outer/inner/\agi\ of $T \colon V \to W$ if and only if
 $G^*$ is an outer/inner/\agi\ of $T^*$.
\end{proposition}

\begin{proof} This follows from the defining equations for these properties.
 For example, if $G$ is an outer inverse of $T$, we have 
$G^*T^*G^* = (GTG)^* = G^*$, and
 the reverse implication follows from the injectivity of the transposition map.
\end{proof}

With the results of this section, we obtain the following duality principle 
for generalized inverses.

\begin{remark} \label{DualPrin}
Given a valid statement for linear maps on arbitrary vector spaces $V$ 
over a common field involving inclusions, $\{0\}$ and $V$, sums and intersections, direct sums, kernels and images,
projectors, and outer/inner/\agi s,
we obtain a valid dual statement by
\begin{itemize}
 \item reversing the order of the linear maps and the corresponding domains and codomains,
 \item reversing inclusions and interchanging $V$ and $\{0\}$,
 \item interchanging sums and intersections,
 \item interchanging kernels and images.
\end{itemize}
\end{remark}

For example, one easily checks that in Proposition \ref{prop:OIChar}, the statements \ref{oi:Nashed5}--\ref{oi:Nashed7} are 
the duals of \ref{oi:Nashed2}--\ref{oi:Nashed4} in this sense. 


\providecommand{\bysame}{\leavevmode\hbox to3em{\hrulefill}\thinspace}
\providecommand{\MR}{\relax\ifhmode\unskip\space\fi MR }
\providecommand{\MRhref}[2]{%
  \href{http://www.ams.org/mathscinet-getitem?mr=#1}{#2}
}
\providecommand{\href}[2]{#2}

\end{document}